\documentclass[10pt]{article}

\usepackage[utf8]{inputenc}

\usepackage{amsmath,graphicx}
\usepackage{amsthm,amssymb,mathtools}
\usepackage{tikz}

\usepackage{fullpage}

\usepackage{nicefrac}
\usepackage{mathrsfs}

\usepackage{array}
\usepackage{diagbox}

\title{Ordinal Sums of Numbers}

\author{Alexander Clow\\
\small Department of Mathematics\\[-0.8ex]
\small Simon Fraser University\\[-0.8ex] 
\small Burnaby, British Columbia, Canada\\
\small\tt alexander\_clow@sfu.ca\\
\and
Neil Anderson McKay\\
\small Department of Mathematics and Statistics\\[-0.8ex]
\small University of New Brunswick\\[-0.8ex]
\small Saint John, New Brunswick, Canada\\
\small\tt neil.mckay@unb.ca\\}

\newcommand{\interval}[2]{\ensuremath{\operatorname{B}\left(#1,#2\right)}}

\newcommand{\os}{\mathbin{:}}

\usepackage{hyperref}
\usepackage[capitalise, noabbrev]{cleveref}

\newtheorem{theorem}{Theorem}

\newtheorem{corollary}[theorem]{Corollary}
\newtheorem{lemma}[theorem]{Lemma}

\newtheorem{proposition}[theorem]{Proposition}

\theoremstyle{definition}

\theoremstyle{remark}

\newcommand{\game}[2]{\{ #1 | #2 \}}

\newcommand{\TT}{\textsc{Teetering Towers}}
\newcommand{\Hackenbush}{\textsc{Hackenbush}}

\begin{document}

\maketitle
\begin{abstract}
In this paper we consider ordinal sums of combinatorial games where each summand is a number, not necessarily in canonical form. In doing so we give formulas for the value of an ordinal sum of numbers where the literal form of the base has certain properties. These formulas include a closed form of the value of any ordinal sum of numbers where the base is in canonical form. Our work employs a recent result of Clow which gives a criteria for an ordinal sum $G\os K = H\os K$ when $G$ and $H$ do not have the same literal form, as well as expanding this theory with the introduction of new notation, a novel ruleset \TT{}, and a novel construction of the canonical forms of numbers in \TT{}. In doing so, we resolve the problem of determining the value of an ordinal sum of numbers in all but a few cases appearing in Conway's \emph{On Numbers and Games}; thus generalizing a number of existing results and techniques including Berlekamp's sign rule, van Roode’s signed binary number method, and recent work by Carvalho, Huggan, Nowakowski, and Pereira dos Santos. We conclude with a list of open problems related to our results.
\end{abstract}

\section{Introduction}

A \emph{combinatorial game} $G$ is a two-player game of no chance and perfect information. The players of a combinatorial game are refereed to as Left  (given female pronouns) and Right (given male pronouns). A game $G$ is often written $G \cong \game{L(G)}{R(G)}$. Here $\cong$ denotes that two games are identical, $L(G)$ is the set of games (options) that Left can move to if she moves first, and $R(G)$ is the set of games that Right can move to if he moves first. We do not insist on it being either player's turn apriori; this allows a range of algebraic structures to emerge in our analysis of games.

Though combinatorial games in general allow for infinite sequences of moves or returning to a previous position, the games we consider in this paper have neither of these properties; that is play will complete after a finite number of turns regardless of the decisions made by either player. Games with this property are called short or finite. In particular we consider Normal Play games; games were a player loses when they are unable to make a move on their turn. 

In this paper we consider two binary operations on games. The more significant of these structures being the abelian group (under normal play) $\mathbb{G}=(\mathcal{G}/=,+)$  where $\mathcal{G} = \{\text{all combinatorial games}\}$, $+$ is the disjoint sum of two games defined by
\[
G+H \cong \game{G+L(H),L(G)+H}{R+R(G),R(H)+G}
\]
where  $G\cong \game{L(G)}{R(G)}$ and $H\cong \game{L(H)}{R(H)}$ and addition of a single game to a set of games is preformed pointwise in the expected manner. Inverses are given by $-G \cong \game{-R(G)}{-L(G)}$ (roles of Left and Right are switched) and $G \leq H$ if Left wins moving second in $G-H$. Thus, $\{\mathcal{G}\}/\!=$ is the set of all games considered modulo equality under this partial order. The class a game belongs to in $\{\mathcal{G}\}/\!=$ is called its value and has a unique simplest representative called the canonical form. For more on canonical forms see \cite{siegel2013combinatorial}.

The group $\mathbb{G}$ is both natural to consider and significant for the following reasons. The binary operation $+$ can be intuitively thought of as placing two game board between each player and allowing players to move on exactly one board per turn. This situation arises naturally in a variety of games where after some sequences of moves the game decomposes into several smaller games, where moving in one subgame does not effect the others. In such situations the winner of the sum is determined by the values of a the summands. Because of this value is the primary invariant analysed when studying combinatorial games.

This paper primarily considers another significant binary operation on games, the ordinal sum. Given games $G\cong \game{L(G)}{R(G)}$ and $H\cong \game{L(H)}{R(H)}$ the ordinal sum of $G$ and $H$, denoted $G\os H$ is defined by,
\[
G\os H \cong \game{L(G),G\os L(H)}{R(G),G\os R(H)}.
\]
Intuitively, one their turn either player can move in $G$ or in $H$, but should they move in $G$, then neither player can move in $H$ for the remainder of play. Ordinal sums are an ideal model for situations where one or both players have the opportunity to delay making critical moves but in such a way that once a critical move is made the ability to delay is no longer available. As the ordinal sum is nonabelian we call the left summand the base and the right summand the subordinate. This is done largely to prevent confusion with the names of each player. 

A key tool when analysing ordinal sums is the Colon Principle which states that if $G\leq H$, then $K\os G \leq K\os H$ for all $K$. Of course this implies that if $G=H$, then $K\os G = K\os H$. Unfortunately the claim, if $G=H$, then $G \os K = H\os K$ is false. Thus, some knowledge of the literal form and not just the value of the base is required. This is particularly significant because in many contexts it is sufficient to consider games in canonical form. The failure of this claim is a major roadblock to studying ordinal sums. As a result ordinal sums have proved very challenging to study in generality. The exceptions to this rule being impartial games, where there is an exact method to calculate the value of an ordinal sum of two impartial games given in \cite{carvalho2018ordinal}, and for numbers where some partial results are known.

To see an example of an ordinal sum appearing in an actual game we introduce the well-known ruleset \textsc{Blue-Red Hackenbush}. \textsc{Blue-Red Hackenbush} is played on a rooted graph $G = (V,E)$ with a $2$-coloured (blue and red) edge set. On her turn Left can choose a blue edge $e$ and delete from the graph. Additionally, every edge this might disconnect from the root vertex is also deleted. Similarly, on his turn Right can delete red edges. See \cref{Fig: Hackenbush}.

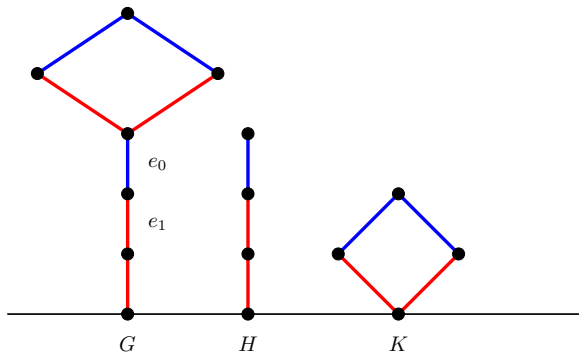
\begin{figure}[h!]
\centering 
\scalebox{0.8}{ 
    \begin{tikzpicture}

 \draw[thick,black] (8.5,0) -- (18,0);

 	\node[fill=none] at (10.5,-0.5) (nodes) {$G$};
	\draw[ultra thick,red](10.5,0)--(10.5,1);
	\draw[ultra thick,red](10.5,1)--(10.5,2);
 	  \node[fill=none] at (11,1.5) (nodes) {$e_1$};
	\draw[ultra thick,blue](10.5,2)--(10.5,3);
 	  \node[fill=none] at (11,2.5) (nodes) {$e_0$};
	\draw[ultra thick,red](10.5,3)--(9,4);
	\draw[ultra thick,blue](9,4)--(10.5,5);
	\draw[ultra thick,blue](10.5,5)--(12,4);
	\draw[ultra thick,red](12,4)--(10.5,3);
 
	\filldraw[black] (10.5,0) circle [radius=1mm];
   	\filldraw[black] (10.5,1) circle [radius=1mm];
   	\filldraw[black] (10.5,2) circle [radius=1mm];
   	\filldraw[black] (10.5,3) circle [radius=1mm];
   	\filldraw[black] (10.5,5) circle [radius=1mm];
   	\filldraw[black] (9,4) circle [radius=1mm];
   	\filldraw[black] (12,4) circle [radius=1mm];

    \node[fill=none] at (12.5,-0.5) (nodes) {$H$};
	\draw[ultra thick,red](12.5,0)--(12.5,1);
	\draw[ultra thick,red](12.5,1)--(12.5,2);
	\draw[ultra thick,blue](12.5,2)--(12.5,3);

    \filldraw[black] (12.5,0) circle [radius=1mm];
   	\filldraw[black] (12.5,1) circle [radius=1mm];
   	\filldraw[black] (12.5,2) circle [radius=1mm];
   	\filldraw[black] (12.5,3) circle [radius=1mm];

    \node[fill=none] at (15,-0.5) (nodes) {$K$};
	\draw[ultra thick,red](15,0)--(14,1);
	\draw[ultra thick,blue](14,1)--(15,2);
	\draw[ultra thick,blue](15,2)--(16,1);
	\draw[ultra thick,red](16,1)--(15,0);

    \filldraw[black] (15,0) circle [radius=1mm];
   	\filldraw[black] (15,2) circle [radius=1mm];
   	\filldraw[black] (14,1) circle [radius=1mm];
   	\filldraw[black] (16,1) circle [radius=1mm]; 
    \end{tikzpicture}
  }
    \caption{Games of \textsc{Blue-Red Hackenbush}.}
    \label{Fig: Hackenbush}
\end{figure}

Notice that in \cref{Fig: Hackenbush} once $e_0$ is removed by Left or $e_1$ is removed by Right every edge above $e_0$ is also deleted. Thus, $G \cong H \os K$. \textsc{Blue-Red Hackenbush} is a classic example of a special class of games called numbers. Numbers are games of the form $G = \game{L(G)}{R(G)}$ such that for all $G^L \in L(G)$ and $G^R \in R(G)$, $G^L < G^R$. Numbers are significant as they form one of the most well-behaved classes of games. One of the main reasons for this is hinted at by their name, as the the values of the games we call numbers are exactly the surreal numbers defined by Conway \cite{conway2000numbers}. In particular, numbers with a finite birthday are exactly the subgroup of the surreals which is isomorphic to the dyadic rationals, denoted $\mathbb{D}= \{\frac{a}{2^p}: a,p \in \mathbb{Z}\}$. In this way we label the values of short numbers by their corresponding element in $\mathbb{D}$. For example $\game{}{} = 0$, $\game{0}{} = 1$, and $\game{0}{1} = \frac{1}{2}$. 

We think about constructing all values of short numbers as short games in the following way. First, for $x\geq 0$ if $x$ is an integer, then $x+1 = \game{x}{}$. Second, if $x$ is not an integer, then $x = \frac{a}{2^p} = \game{\frac{a}{2^p} - \frac{1}{2^p}}{\frac{a}{2^p} + \frac{1}{2^p}}$ \cite{siegel2013combinatorial}. Supposing $a$ is odd, $\frac{a}{2^p} - \frac{1}{2^p}$ and $\frac{a}{2^p} + \frac{1}{2^p}$ can both be expressed with a denominator $2^q$ where $q<p$. Thus, we can construct new numbers from those we have already generated in a similar manner to how reals are constructed using Dedekind cuts. The difference here is that rather than beginning with rationals to form reals, we take a list of numbers we have generated, then generate the next number by deciding where the new number fits in the total ordering (i.e which of the existing numbers it is bigger than and which it is smaller than). To construct negative numbers we negate positive numbers. For more on short numbers see Section~2.2. For a broader study of numbers see \cite{conway1996surreal, conway2000numbers, ehrlich2001number, lurie1998effective, siegel2013combinatorial}.

Importantly for our purposes in \cite{conway2000numbers} Conway constructs the canonical forms of surreals as sign sequences, which correspond to stalks (paths where the rooted vertex is a leaf) in \textsc{Blue-Red Hackenbush}. As we have already seen these are themselves repeated ordinal sums of integers in canonical form. For a detailed description of this see \cite{conway1996surreal}. Determining the values of such positions is well-studied. In particular we highlight Berlekamp's Sign Rule and van Roode’s signed binary number method which are both described in \cite{albert2019lessons} (see pages 134 and 135). Unfortunately neither of these methods lead to closed expressions. Of particular interest to our work is the van Roode’s signed binary number method as much of our work in Section~4 can be looked at as a generalization of this method to position where the base of the sum is not necessarily in canonical form. Ironically, through this generalization we arrive at a closed form for the case where the base of an ordinal sum of numbers is in canonical form (which is exactly what van Roode’s method describes). Along with these more classical methods, the problem of determining the value of an ordinal sum of numbers was recently considered by Carvalho et al. \cite{carvalho2022ordinal} who give a generalization of van Roode’s method for some ordinal sums where the base in an integer in which only one player has an option.

As a means to examine more complicated ordinal sums of numbers we introduce a novel ruleset which we call \TT{}. We introduce a notation for repeated ordinal sums. Let  $ \bigodot_{i = 1}^{n-} G_i = G_1 \os G_2 \os \cdots \os G_{n}$ for any collection of games $G_1,\ldots, G_{n}$. Using this notation we say a \emph{tower} $T$ is a game of the of the form $T =  \bigodot_{i = 1}^{n} (b_i + r_i)$ where for all $i$, $b_i$ is a non-negative integer in canonical form and $r_i$ is a non-positive integer in canonical form. Each $b_i+r_i$ is called a story of $T$. A position in \TT{} is a game of the form $ \sum_{j=1}^m \bigodot_{i=1}^n (b_{j,i}+r_{j,i})$ where $\sum$ is used in the standard way where $+$ is the disjunctive sum. For an example of a game of \TT{} see \cref{SumOfTowers}.

\begin{figure}[ht]
 \centering
\scalebox{0.5}{ 
	    \begin{tikzpicture}
     
\draw[thick,black] (0,0) -- (4.25,0);
 
	\fill[blue!70] (0,0.25) rectangle (2,1.25);
	
	\fill[red!70] (2.25,0.25) rectangle (4.25,1.25);
	\fill[red!70] (2.25,1.5) rectangle (4.25,2.5);
	
	\fill[black] (1,1.4) rectangle (1.2,2.5);
	 
\draw[thick,black] (0,2.65) -- (4.25,2.75);

	\fill[blue!70] (0,3) rectangle (2,4);
	\fill[blue!70] (0,4.25) rectangle (2,5.25);
 
	\fill[red!70] (2.25,3) rectangle (4.25,4);
	\fill[red!70] (2.25,4.25) rectangle (4.25,5.25);
	
\draw[thick,black] (6,0) -- (10.25,0);
 
	\fill[blue!70] (6,0.25) rectangle (8,1.25);
	\fill[blue!70] (6,1.5) rectangle (8,2.5);
	
	\fill[red!70] (8.25,0.25) rectangle (10.25,1.25);
	\fill[red!70] (8.25,1.5) rectangle (10.25,2.5);
	 
\draw[thick,black] (6,2.75) -- (10.25,2.75);

	\fill[blue!70] (6,3) rectangle (8,4);
	\fill[blue!70] (6,4.25) rectangle (8,5.25);
	\fill[blue!70] (6,5.5) rectangle (8,6.5);
 
	\fill[red!70] (8.25,3) rectangle (10.25,4);
	\fill[red!70] (8.25,4.25) rectangle (10.25,5.25);
	
	\fill[black] (9.25,5.45) rectangle (9.45,6.5);
	
\draw[thick,black] (6,6.75) -- (10.25,6.65);

	\fill[red!70] (8.25,7) rectangle (10.25,8);
	
\draw[thick,black] (12,0) -- (16.25,0);

	\fill[blue!70] (12,0.25) rectangle (14,1.25);
	\fill[blue!70] (12,1.5) rectangle (14,2.5);
	\fill[blue!70] (12,2.75) rectangle (14,3.75);
	\fill[blue!70] (12,4) rectangle (14,5);
	\fill[blue!70] (12,5.25) rectangle (14,6.25);
	
	\fill[red!70] (14.25,0.25) rectangle (16.25,1.25);
	\fill[red!70] (14.25,1.5) rectangle (16.25,2.5);
	\fill[red!70] (14.25,2.75) rectangle (16.25,3.75);	 
    \end{tikzpicture}
  }
\caption{A game of \textsc{Teetering Towers}. On her turn Left can choose a blue brick and remove it from the tower. Similarly, on his turn Right removes red bricks. When a brick is removed every story above the one from which the brick was removed from is also removed. \label{SumOfTowers}}
\end{figure}
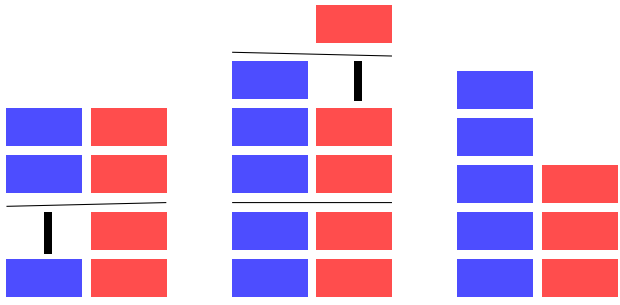

Notice that any position in \TT{} where each story is monochromatic (i.e $b_i = 0$ or $r_i = 0$) is exactly a \textsc{Blue-Red Hackenbush} stalk. For example, see \cref{HackenbushTowers}. Hence, forms (and thus also values) from \Hackenbush\ strings arise in \TT{}. There are forms in \TT{} that do not occur in \Hackenbush; we show this explicitly in  \cref{sec:towers}. This relationship is deliberate as it allows for games (numbers) to be constructed which are far from being in canonical form, while also allowing for well understood positions to appear. Given known criteria, such as those that appear in \cite{carvalho2022note}, it is easy to determine that every position in \TT{} is a number.

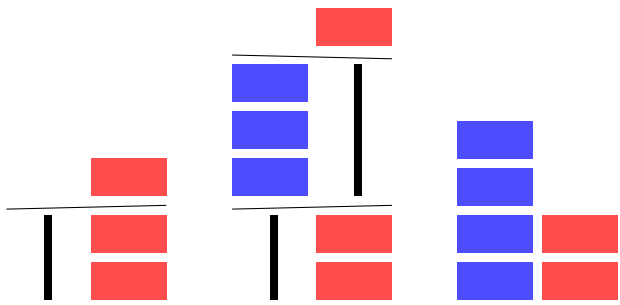
\begin{figure}[!h]
 \centering
\scalebox{0.5}{ 
	    \begin{tikzpicture}
     
\draw[thick,black] (0,0) -- (4.25,0);
 
	\fill[red!70] (2.25,0.25) rectangle (4.25,1.25);
	\fill[red!70] (2.25,1.5) rectangle (4.25,2.5);
	
	\fill[black] (1,0.25) rectangle (1.2,2.5);
	 
\draw[thick,black] (0,2.65) -- (4.25,2.75);

	\fill[red!70] (2.25,3) rectangle (4.25,4);
	
\draw[thick,black] (6,0) -- (10.25,0);
 
	\fill[black] (7,0.25) rectangle (7.2,2.5);
	
	\fill[red!70] (8.25,0.25) rectangle (10.25,1.25);
	\fill[red!70] (8.25,1.5) rectangle (10.25,2.5);
	 
\draw[thick,black] (6,2.65) -- (10.25,2.75);

	\fill[blue!70] (6,3) rectangle (8,4);
	\fill[blue!70] (6,4.25) rectangle (8,5.25);
	\fill[blue!70] (6,5.5) rectangle (8,6.5);
 
	\fill[black] (9.25,3) rectangle (9.45,6.5);
	
\draw[thick,black] (6,6.75) -- (10.25,6.65);

	\fill[red!70] (8.25,7) rectangle (10.25,8);
	
\draw[thick,black] (12,0) -- (16.25,0);

	\fill[blue!70] (12,0.25) rectangle (14,1.25);
	\fill[blue!70] (12,1.5) rectangle (14,2.5);
	\fill[blue!70] (12,2.75) rectangle (14,3.75);
	\fill[blue!70] (12,4) rectangle (14,5);
	
	\fill[red!70] (14.25,0.25) rectangle (16.25,1.25);
	\fill[red!70] (14.25,1.5) rectangle (16.25,2.5);
    \end{tikzpicture}
  }
\caption{A game of \textsc{Teetering Towers} that is equivalent to a \Hackenbush{} 
position.\label{HackenbushTowers}}
\end{figure}

The paper is structured as follows. In Section~2 we establish ideas and notation that are key to the rest of the paper. In Section~2 we cover equivalence modulo domination as introduced in \cite{Clow2022} and its implications for ordinal sums, some important background on numbers, and novel notation which is convenient when considering numbers modulo domination. In Section~3 we demonstrate how the contents of Section~2 can be applied to the ruleset \TT{}. In doing so we construct a tower equivalent modulo domination to the canonical form of any number $x$ as an ordinal sum, but unlike in Conway's \textsc{Hackenbush} stalk/sign sequence model \cite{conway1996surreal, conway2000numbers}, each of our summands has the same sign as $x$ or is equal to $0$ (see \cref{LeaningTower}). Finally, in Section~4 we focus on determining the value of an ordinal sum of numbers for various literal forms as the base of the sum. The primary results of which is the quite general \cref{GameOsInteger}, which gives a powerful formula for the ordinal sum of a number (base) and an integer (subordinate), as well as \cref{BalancedOSNumber} which is a closed-form expression for the value of an ordinal sum of a balanced number  with radius $\frac{-1}{2^p}$ (base) and any number as the subordinate. Using \cref{BalancedOSNumber} we give a formula for the ordinal sum of any numbers where the base is in canonical form, see Table~3.

\section{Preliminaries}

\subsection{Equivalence Modulo Domination}
The main tool that enables our analysis of ordinal sums is equivalence modulo domination.  We say $G$ and $H$ are \emph{equivalent modulo domination} and we write $G \triangleq H$, if for all options of $G +(-H)$ by the first player there is a winning response by the second player in the other summand. Although quite natural to the knowledge of the authors, the concept of equivalence modulo domination first appears in the recent thesis of Clow \cite{Clow2022}. 

By definition, $G \triangleq H$ implies $G=H$ as $G \triangleq H$ implies for every move by the first player there is a winning response by the second player in $G-H$. Also by definition, for every incentive in $G$ for the first player, there is an incentive that is equally appealing for the second player in $-H$ (and visa versa). We use $\triangleq$ as the notion for equivalence modulo domination due to this second fact, as incentives are often denoted by $\Delta$ in the literature \cite{siegel2013combinatorial}. 

Notice that $G = H$ does not necessarily imply that $G\triangleq H$. For example, in the difference  $\game{-1}{1} - 0 = 0$ the only winning response for the second player is in the same summand as the first player. Thus, $\game{-1}{1} = 0$ but $\game{-1}{1} \not\triangleq 0$. Hence, equivalence modulo domination is a refinement of equality. This is natural given two games are equal if and only if they have the same canonical form, where canonical forms are achieved by removing dominated and reversible options. Meanwhile, equivalence modulo domination can be thought of as removing dominated options but not reversible options. For a formal proof of this see Lemma~\ref{TeqImpDomRed}.

\begin{lemma}\label{TeqImpDomRed}
If $H$ is derived from $G$ by removing a dominated option, then $G \triangleq H$.
\end{lemma}

\begin{proof}
Let $H$ be derived from $G$ by removing a dominated option, say $G'$. Consider $G - H$. If the first 
player moves in $G$ to $G'$, the opponent responds in $H$ to the option that dominated $G'$. 
Otherwise, the second player can always respond in the other summand to the same option.
\end{proof}

\begin{theorem}[Clow \cite{Clow2022}]\label{TeqBaseImpOSTeq}
If $G \triangleq H$, then for all games $K$, $G \os K \triangleq H \os K$ and $K \os G \triangleq K \os H$.
\end{theorem}

\begin{proof}
Let $G \triangleq H$. Consider $G \os K - H \os K$. Suppose without loss of generality Left moves 
first in $G \os K$. If her move was in the subordinate Right mirrors in $-H \os K$, whereas if her 
move was in the base, then Right responds in $-H$ as if playing the sum $G-H$. As $G \triangleq H$, 
there exists such moves that are winning. Thus, playing $G \os K - H \os K$ if Left (Right) moves 
first on either summand, then Right (Left) has a winning response of the other. Hence, $G \os K 
\triangleq H \os K$ as required. Observe that $K \os G \triangleq K \os H$ follows by a similar mirroring strategy. 
\end{proof}

\begin{corollary}\label{TeqBaseImpOSeq}
    If $G \triangleq H$, then for all games $K$, $G \os K = H \os K$.
\end{corollary}

The converses of \cref{TeqBaseImpOSTeq,TeqBaseImpOSeq} are false, as we see in the following 
example; $\displaystyle \game{0}{3} = 1 = \game{\tfrac{1}{2}}{3}$ and $\game{0}{3} 
\not\triangleq \game{\frac{1}{2}}{3}$, yet 
$\game{0}{3} \os 1 \triangleq \game{1}{3} \triangleq \game{\tfrac{1}{2}}{3} \os 1$. Also of note is that \cref{TeqBaseImpOSTeq} implies \cref{McKay's Theorem}, a known result which deals with canonical forms in the base of an ordinal sum.

\begin{corollary}[\cite{conway2000numbers, mckay2016forms}]\label{McKay's Theorem}
    If $G$ has no reversible options and $K$ is the canonical form of $G$, then for all $H$, $G \os H = K \os H$.
\end{corollary}

\begin{proof}
    If $G$ has no reversible options then either $G$ is its own canonical form ($G\cong K$) at which point the statement is trivial or $G$ has some set of dominated options, whose removal would result in $K$. Thus, \cref{TeqImpDomRed} implies $G \triangleq K$. The result follows immediately.
\end{proof}

\subsection{Fundamental Results About Numbers}

In this subsection we cover some facts about numbers in relation to ordinal sums. We encourage readers to spend some time with this short section as many of the fact presented here are critical to the rest of the paper. 

First, recall that $G$ is a number if $G^L < G < G^R$ for all $G^L \in L(G)$ and $G^R \in R(G)$. We note the following proposition which is well known in the literature. 

\begin{proposition}[Simplicity Theorem]
    If $G = \game{a}{b}$ where $a<b$ are numbers, then $G = x$ where $x$ is the unique number with the smallest birthday on the open interval $(a,b)$.
\end{proposition}

A proof is omitted here, but can be found in \cite{siegel2013combinatorial}. Next, consider \cref{OptionDominationInOrdSumsOfNumbers}.

\begin{proposition}\label{OptionDominationInOrdSumsOfNumbers}
Let $G$ be a number whose options are numbers and $H$ be any game. If $L(H) \neq \emptyset$, then 
$G^L$ is strictly dominated in $G \os H$.
\end{proposition}

\begin{proof}
We claim for any $H^L \in L(H)$, $G^L - G\os H^L < 0$. Consider the difference $G^L - G\os H^L < 0$. Right wins moving first by moving $-G\os H^L$ to $-G^L$. If Left moves first, then she must play in $G^L$, or $-G\os H^L$. If Left moves $G^L$ to some $G^{LL}$, then $G^{LL}< G^L$ as $G^L$ is a number, implying that Right moving $-G\os H^L$ to $-G^L$ is a winning response for Right. Similarly, if Left moves is the subordinate of $-G\os H^L$, the Right wins moving in the basis to $-G^L$. Finally, if Right moves $-G\os H^L$ to $-G^R$, then this is losing as $G$ is a number implies $G^L < G < G^R$, hence, $G^L-G^R < 0$.
\end{proof}

\begin{corollary}\label{OS1Dom}
If $G$ is number whose options are numbers, then $G \os 1 \triangleq \game{G}{R(G)}$. Similarly $G 
\os -1 \triangleq \game{L(G)}{G}$.
\end{corollary}

\begin{proposition}\label{Teq=Number}
If the options of $G$ are numbers and both players have at 
least one option, then $G \triangleq \game{a}{b}$ for some numbers $a$ and $b$ in canonical form.
\end{proposition}

\begin{proof}
This follows from \cref{TeqImpDomRed} and that numbers are totally ordered (comparable).
\end{proof}

\subsection{Balls}

In this section we define a special class of games we call balls due to their similarity to intervals in $\mathbb{R}$. We care about these because all numbers whose options are numbers can be reduced to a ball under equivalence modulo and special balls are well behaved summands in an ordinal sum.

We begin by noticing that many games with two options, such as non-integer numbers in canonical form, have the property that the options have a common difference with some game between them. That is if $\frac{a}{2^p}$ where $a = 2b+1$ is in canonical form, then $\frac{a}{2^p} \cong \game{\frac{b}{2^{p-1}}}{\frac{b+1}{2^{p-1}}}$. With this in mind, we say a game $\game{x}{y}$ where $x$ and $y$ are in canonical form is a \emph{ball} of radius $\Delta$ centered at $m$ if $x = m+\Delta$ and $y = m-\Delta$ where $m$ and $\Delta$ are games. We denote such a game by $\interval{m}{\Delta}$ where $m$ is the \emph{midpoint} and $\Delta$ is the \emph{radius}. Thus, the canonical form of $\frac{a}{2^p}$ is $ \interval{\frac{a}{2^p}}{-\frac{1}{2^p}}$.

\begin{table}[h!]
\centering
\begin{tabular}{lcl}
\text{Literal form, $G$} & & Description\\ \hline 
$\game{1}{3*}$ & $G = 2$ & a number \\
$\game{1}{3, 4}$ & $G^{L}< G^{R}$ & a number whose options are in canonical form \\
$\game{1}{5}$ & $1 = 3-2$, $5 = 3 + 2$ & an interval number \\
$\game{1}{3}$ & $1 = \game{1}{3} - 1$, $3 = \game{1}{3} + 1$ & a balanced interval number \\
$\game{1}{}$ & $ \game{1}{} \cong 2$ & a number in canonical form \\
\end{tabular}
\label{tab:number}
\caption{Some forms of the value $2$}
\end{table}

In this paper, we only consider $\Delta$ and $m$ and $\interval{m}{\Delta}$ to be numbers. This implies $\Delta < 0$. Unless otherwise stated assume that every game given using the notation $\interval{m}{\Delta}$ is a number.

A ball $G \cong \game{m+\Delta}{m-\Delta}$ is \emph{balanced} if $G + G = m + m$. Equivalently, a ball $G \cong \game{a}{b}$ is \emph{balanced} if $G + G = a + b$. We say a game (not necessarily a ball) is balanced if it is equivalent modulo domination to a balanced ball. Examples of balanced balls are show in Table~2. That is a game is balanced when its value is the mean of its best options. Rather than giving the definition in this informal way, we insist on the former definition in service of giving a definition that remains well defined when $m,\Delta$ are general games (where division is not defined).

\begin{table}[h!]
	\centering
	$\begin{array}{lcl}
		\text{Interval game} & \text{Literal form} & \text{Value}\\ \hline 
		\interval{\frac{3}{4}}{-\frac{1}{4}} & \game{\frac{1}{2}}{1} & \frac{3}{4}\\
		\interval{2}{-1}					 & \game{1}{3} & 2\\
		\interval{0}{-1} 					& \game{-1}{1} & 0 \\
		\interval{0}{1} 					&  \game{1}{-1} & \pm 1 \\
	\end{array}$
	\label{tab:balls}	
	\caption{Some balanced balls.}
\end{table}

\begin{lemma}\label{BalcendImpliesEQ}
If $m$ and $\interval{m}{\Delta}$ are numbers and $\interval{m}{\Delta}$ is balanced, then $\interval{m}{\Delta} = m$.
\end{lemma}

\begin{proof}
As the game is balanced, from the definition $\interval{m}{\Delta} + \interval{m}{\Delta} = m + m$. As this implies  $\interval{m}{\Delta} -  m = -(\interval{m}{\Delta} -m)$ which is a number, it must be that $\interval{m}{\Delta} -  m = 0$. Hence $\interval{m}{\Delta} = m$ as required.
\end{proof}

\begin{proposition}\label{prop: canonical radius}
If $\interval{m}{\Delta}$ is balanced and $m = \frac{a}{2^p} \neq 0$ where $a$ is balanced or $p=0$, 
then $ -\frac{1}{2^p} \le \Delta < 0$.
\end{proposition}
\begin{proof}
Recall that $\Delta < 0$ is assumed. So if the statement is false for some balanced $\interval{m}{\Delta}$, then $\Delta < -\frac{1}{2^p}$. Recall that $\interval{m}{\Delta} \cong \game{m + \Delta}{m-\Delta} = x$ where $x$ is the number with the smallest birthday on the interval $(m + \Delta, m - \Delta)$. Thus, if $\Delta < -\frac{1}{2^p}$, then $\frac{a-1 }{2^p}, \frac{a +1 }{2^j} \in (m + \Delta, m - \Delta)$ at least one of which has a smaller birthday than $m = \frac{a}{2^j}$. Then $\interval{m}{\Delta} \neq m$ contradicting the fact that  $\interval{m}{\Delta}$ is balanced implies $\interval{m}{\Delta} = m$ by \cref{BalcendImpliesEQ}. Given this contradiction, it must be the case that $ -\frac{1}{2^p} \le \Delta < 0$ as required.
\end{proof}

To formalize the language of this proposition, we say that the \emph{canonical radius} of number is $1$ if the game is an integer and the radius of its canonical form otherwise. Then the canonical radius of $m = \frac{n}{2^j}$ is $ -\frac{1}{2^j}$. This notation is adopted as \cref{prop: canonical radius} has shown that the radius of a balanced number is bounded above by its canonical radius.

\section{Numbers as Towers}\label{sec:towers}

Forms, and thus also values, from \Hackenbush strings arise in \TT{}. However, there are forms in \TT{} that do not occur in \Hackenbush. An example is \cref{SumOfTowers}. \cref{LeaningTower} constructs a family of examples.

In this section we describe the form that games in \TT{} can take, in particular the forms of a single tower. As recovering the form of a sum of games whose options are well understood is not difficult, this is sufficient to give a general description of the forms of a position in \TT.

\begin{lemma}\label{n-nTeq1-1}
Let $n$ be any non-zero integer. Then $n - n \triangleq 1 - 1 \cong \game{-1}{1}$.
\end{lemma}

\begin{proof}
Consider the difference $n - n$. Then, \[ n - n \cong 
\game{n-1}{} + \game{}{-n + 1} \cong \game{n - 1 - n }{ n - n + 1} \triangleq \game{-1 }{ 1} \cong 
1 - 1.\] This is the desired result.
\end{proof}

Due to this the game $1 - 1$ is surprisingly significant for teetering towers, particularly those of value $0$. Thus, 
we denote the game $1 - 1$ by the special name $\underline{0}$.
 By the Colon Principle it is ease to see that towers of value $0$ are exactly those where each story has value $0$.

\begin{lemma}\label{n-mTeq0+k}
Let $n,m$ be non-zero integers such that $n > m$. If $n - m = k$, then $n - m \triangleq \game{k-1}{k + 1} \triangleq k + \underline{0} $.
\end{lemma}

\begin{proof}
The proof is similar to the proof of \cref{n-nTeq1-1},
\[ n - m \cong \game{n-1}{} + \game{}{-m + 1} \cong \game{n - 1 - m }{ n - m + 1} \triangleq 
\game{k-1 }{k + 1} \triangleq \game{k-1}{} + \game{-1}{1} \cong k + \underline{0}.\]
Implying the desired equality.
\end{proof}

For convenience we adopt the notation $\underline{k} \cong k+ \underline{0}$. Notice that every story of a teetering tower is an integer $k$ in canonical form or is equivalent modulo domination to $\underline{k}$ for some $k$. Given this it is useful to consider the following ordinal sum.

\begin{lemma}\label{underline0 on Top}
    If $\interval{m}{-\frac{1}{2^p}}$ is balanced, then 
    \[
    \interval{m}{-\frac{1}{2^p}}: \underline{0} \triangleq \interval{m}{-\frac{1}{2^{p+1}}}.
    \]
    which is also balanced.
\end{lemma}

\begin{proof}
    We compute the given sum directly,
\begin{center}
\begin{align*}
    \interval{m}{-\frac{1}{2^p}}: \underline{0} \triangleq \game{\interval{m}{-\frac{1}{2^p}}: -1}{ \interval{m}{-\frac{1}{2^p}}: 1}\\
    \triangleq \game{\game{m-\frac{1}{2^p}}{m}}{\game{m}{m+\frac{1}{2^p}}}\\
    \triangleq \game{\interval{m-\frac{1}{2^{p+1}}}{-\frac{1}{2^{p+1}}}}{\interval{m+\frac{1}{2^{p+1}}}{-\frac{1}{2^{p+1}}}}\\
    \triangleq \game{m-\frac{1}{2^{p+1}}}{m+\frac{1}{2^{p+1}}}
    \triangleq \interval{m}{-\frac{1}{2^{p+1}}}
\end{align*}
\end{center}
as required. Observe that $\interval{m}{-\frac{1}{2^{p+1}}} = m$ as $m$ is the simplest number on the interval $(m- \frac{1}{2^p}, m+ \frac{1}{2^p})$ by assumption that $\interval{m}{-\frac{1}{2^p}}$ is balanced, which implies that $m$ is the simplest number on $(m- \frac{1}{2^{p+1}}, m+ \frac{1}{2^{p+1}})$ as well. Now $\interval{m}{-\frac{1}{2^{p+1}}} = m$ implies $\interval{m}{-\frac{1}{2^{p+1}}}$ is balanced by definition.
\end{proof}

Notice that if $\interval{m}{-\frac{1}{2^p}}$ is not balanced, then we cannot suppose $\interval{m}{-\frac{1}{2^{p}}}\os -1 = m-\frac{1}{2^{p+1}}$ or $\interval{m}{-\frac{1}{2^{p}}}\os 1 = m+\frac{1}{2^{p+1}}$. As a result the assumption of balancedness in Lemma~\ref{underline0 on Top} cannot be relaxed.

In \Hackenbush{} all numbers appear as stalks and this construction relies on edges of both colors. Thus, there must be summands which are both positive and negative to form each number as a \Hackenbush\ stalk. In a tower, each non-zero story has a player with more of their color. We can build a tower of any given value where each story has at least as many bricks for the same player.

\begin{theorem}\label{LeaningTower}
Let $x \in \mathbb{D}$ where $x \geq 0$, then there exists a tower $T = x$ where for all stories 
$S$ of $T$, $S \geq 0$.
\end{theorem}

\begin{proof}
We prove this by induction on the exponent in the denominator of the value, but we are showing that there is a tower that is equivalent modulo domination to ball centered at the value with its canonical radius.
When $x\geq0$ is an integer take $T \cong \underline{x}\triangleq \interval{x}{-1} = x$. Suppose that $x = \frac{a}{2^p}$ is a smallest counter example; if $y = \frac{b}{2^q}$ where $q<p$ there is a tower $T_y$ which is equivalent modulo domination to $\interval{y}{r}$ where $r$ is the canonical radius of $y$, such that every story of $T_y$ is non-negative.

As $x$ is a smallest counter example, $x$ is not an integer. Then canonical form of $x$ is $\game{x - 2^{-p}}{x+ 2^{-p}} \cong \interval{x}{-\frac{1}{2^p}}$. 
Notice that $z = x - 2^{-p} = \frac{a-1}{2^p} = \frac{c}{2^{p-l}}$ for some odd number $c = \lfloor \frac{a}{2^l} \rfloor$ if $p-l > 0$ and $z = \lfloor x \rfloor$ otherwise. By the minimality of $x$, there is a tower $T_z$ equivalent modulo domination to $\interval{z}{r}$ where $r = \frac{1}{2^{p-l}}$ is the canonical radius of $z$ such that every story of $T_z$ is non-negative.

We claim $T \cong T_z:(\bigodot_{i=1}^{l-1} \underline{0}):1 = x$. Observe that by induction and Lemma~\ref{underline0 on Top},
\begin{center}
\begin{align*}
    T \triangleq \interval{z}{r}:(\bigodot_{i=1}^{l-1} \underline{0}):1
    \triangleq \interval{z}{\frac{r}{2^{l-1}}}:1 \triangleq \interval{z}{-\frac{1}{2^{p-1}}}:1
   \cong \game{z- \frac{1}{2^{p-1}}}{z+ \frac{1}{2^{p-1}}} :1\\
    \triangleq \game{z}{z+ \frac{1}{2^{p-1}}}
    \triangleq \game{x-\frac{1}{2^p}}{x-\frac{1}{2^p}+ \frac{1}{2^{p-1}}}
    \triangleq \game{x - \frac{1}{2^p}}{x+\frac{1}{2^p}} 
    \cong \interval{x}{-\frac{1}{2^p}} = x
\end{align*}
\end{center}

This completes the proof.
\end{proof}

\begin{corollary}\label{canonicalTower}
If $x \in \mathbb{D}$ is a non-integer in canonical form, then there exists a tower $T \triangleq x$ where for all stories 
$S$ of $T$, $S \cong \underline{0}$ or $S \cong 1$ or $S \cong \underline{k}$ ($k>0$). Here $S \cong \underline{k}$ only if $S$ is the bottom story of the tower $T$.
\end{corollary}

Aside from being quite a nice aesthetic result \cref{canonicalTower} has some piratical applications. In particular when analysing playing proofs involving ordinal sums (of the form $x\os y$) can be tedious if $y$ is expressed as a long \textsc{Hackenbush} string. This is because it is often convenient to consider $a \os b \cong x \os y$ where $a = x\os c$ and $y = c \os b$ rather than $x$ and $y$ directly and when $y$ is a \textsc{Hackenbush} string it can often be the case that the sign of $b$ switches erratically as play progresses. This will not be the case if $y$ is a tower in the form of \cref{canonicalTower}, as each story (except perhaps the bottom story) has value $0$ or $1$ and whenever Right moves in a story of value $0$, the value of that story becomes $1$. As a result, there are fewer cases to consider when analysing such games.

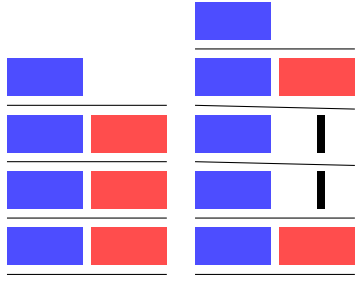
\begin{figure}[h!]
 \centering
\scalebox{0.5}{ 
    \begin{tikzpicture}
     
\draw[thick,black] (0,0) -- (4.25,0);
 
	\fill[blue!70] (0,0.25) rectangle (2,1.25);

	\fill[red!70] (2.25,0.25) rectangle (4.25,1.25);

\draw[thick,black] (0,1.5) -- (4.25,1.5);	

	\fill[blue!70] (0,1.75) rectangle (2,2.75);

	\fill[red!70] (2.25,1.75) rectangle (4.25,2.75);

\draw[thick,black] (0,3) -- (4.25,3);	

	\fill[blue!70] (0,3.25) rectangle (2,4.25);

	\fill[red!70] (2.25,3.25) rectangle (4.25,4.25);

\draw[thick,black] (0,4.5) -- (4.25,4.5);	

	\fill[blue!70] (0,4.75) rectangle (2,5.75);

 %%% Break in diagram

 \draw[thick,black] (5,0) -- (9.25,0);
 
	\fill[blue!70] (5,0.25) rectangle (7,1.25);

	\fill[red!70] (7.25,0.25) rectangle (9.25,1.25);
 
 \draw[thick,black] (5,1.5) -- (9.25,1.5);	

 	\fill[blue!70] (5,1.75) rectangle (7,2.75);

  	 \fill[black] (8.25,1.75) rectangle (8.45,2.75);
	 
\draw[thick,black] (5,3) -- (9.25,2.9);

 	\fill[blue!70] (5,3.25) rectangle (7,4.25);

  	 \fill[black] (8.25,3.25) rectangle (8.45,4.25);
	 
\draw[thick,black] (5,4.5) -- (9.25,4.4);

	\fill[blue!70] (5,4.75) rectangle (7,5.75);

	\fill[red!70] (7.25,4.75) rectangle (9.25,5.75);

\draw[thick,black] (5,6) -- (9.25,6);	

	\fill[blue!70] (5,6.25) rectangle (7,7.25);
    \end{tikzpicture}
  }
\caption{A tower with value $\frac{1}{8}$ and a tower with value $\frac{13}{16}$ as in \cref{canonicalTower}.\label{fig:AllPositiveTower}}
\end{figure}

\section{Values of Ordinal Sums}

Unlike Section~\ref{sec:towers} this section is not concerned with a particular rule set. Rather we are interested in determining the values of ordinal sums independent of the rule set where they arise. A natural first case to consider is ordinal sums equal to $0$. Notice that as numbers are totally ordered the Colon Principle implies that an ordinal sum of numbers is equal to $0$ if and only if each summand is equal to $0$. Hence, we will focus primarily on non-zero ordinal sums as sums equal to $0$ are easy to characterize. But before moving on we note the following example of an ordinal sum with value $0$ as it provides an important insight about balanced games and is also a good exercise in computing ordinal sums for the reader; 

\begin{center}
\begin{align*}
\game{-1/2}{1} \os \game{-1}{1/2} \triangleq \game{\game{-1/2}{1}:-1}{\game{-1/2}{1}:\frac{1}{2}} \\
\triangleq \game{\game{-\frac{1}{2}}{0}}{\game{-1/2}{1}:\game{0}{1}}\\
\triangleq \game{-\frac{1}{4}}{\game{\game{-1/2}{1}}{\game{-1/2}{1}:1} }\\
\triangleq \game{-\frac{1}{4}}{\game{0}{\game{0}{1}}}\\
\triangleq \game{-\frac{1}{4}}{\game{0}{\frac{1}{2}}}\\
\triangleq \game{-\frac{1}{4}}{\frac{1}{4}} \cong \interval{0}{-\frac{1}{4}}
\end{align*}
\end{center}

Observe that neither summand is balanced but the sum is balanced. In general we can also say that the ordinal sum of two balanced numbers need not be balanced. For example $\frac{1}{2} \os \underline{1}$ is not balanced. We now proceed to the primary result of this section.

\begin{lemma}\label{BalancedNumbersColon1}
Let $G \cong \game{a}{b}$ such that $a,b,G$ are numbers and $b-G \leq 1$. If $p>0$ is the least integer such that $\frac{1}{2^p}< b-G$, then 
\[
G\os 1 = G + \frac{1}{2^p}.
\]
Similarly whenever $G-a \leq 1$, $G\os -1 = G - \frac{1}{2^q}$, where $q>0$ is the least integer satisfying $\frac{1}{2^q}< G-a$.
\end{lemma}

\begin{proof}
Let $G = \game{a}{b}$ such that $a,b,G$ are numbers. Recall that $G\os 1 = x$ where $x \in (G,b)$ is the simplest number on the interval from $G$ to $b$. As $b-G \leq 1$ and $b-G$ is not an infinitesimal there is a $p>0$ such that $\frac{1}{2^k}< b-G$. Then by the least integer principle there is a least integer $p>0$ such that $\frac{1}{2^p}< b-G$. Then $b-G \leq \frac{1}{2^{p-1}}\leq  r \leq 1$ where $r$ is the canonical radius of $G$. 

Note that $G\os 1 \triangleq \game{G}{G+(b-G)} = \game{G}{G+\frac{1}{2^p}}$ (see \cref{OS1Dom}). We claim that $\game{G}{G + \frac{1}{2^{p-1}}} = G + \frac{1}{2^p}$. Consider the difference, 
\[
\game{G}{G + \frac{1}{2^{p-1}}} - G - \frac{1}{2^p}.
\]
If Right moves first they may move $\game{G}{G + \frac{1}{2^{p-1}}}$ to $G+\frac{1}{2^{p-1}}$, $-G$ to $-a$ or $-\frac{1}{2^{p}}$ to $0$. All of these moves are obviously losing. Similarly if Left moves first, then they may move $\game{G}{G + \frac{1}{2^{p-1}}}$ to $G$ or $-G$ to $-b$  or $-\frac{1}{2^p}$ to $-\frac{1}{2^{p-1}}$. These are also losing. So $\game{G}{G + \frac{1}{2^{p-1}}} - G - \frac{1}{2^p} = 0$ implying $\game{G}{G + \frac{1}{2^{p-1}}} = G + \frac{1}{2^p}$. Equivalently $G + \frac{1}{2^p}$ is the simplest number on the interval $(G, G + \frac{1}{2^{p-1}})$. As $(G,b) \subseteq (G, G + \frac{1}{2^{p-1}})$, $x = G + \frac{1}{2^p}$ as required. The case of $G\os -1$ follows by a nearly identical argument.
\end{proof}

If we relax our assumption that $b - G \leq 1$ then \cref{BalancedNumbersColon1} is false. As an example consider $G = \game{0}{4} = 1$ and the sum \[G \os 1 = 2 \neq 3 = G + \frac{1}{2^{(-1)}}.\] 
The important thing for our purposes is that games such as this do not often appear in familiar rulesets and are somewhat artificial when you consider numbers as being constructed from ordinal sums.

\begin{theorem}[General Ordinal Sum Theorem]\label{GameOsInteger}
Let $k>0$ be an integer and $G \cong \game{a}{b}$ be numbers. If $b-G\leq 1$ and $\alpha_i>0$ is the least integer such that $2^{-\alpha_i}< b-G - (\sum_{j=1}^{i-1} 2^{-\alpha_j})$, then 
\[
G\os k = G + \sum_{i=1}^k \frac{1}{2^{\alpha_i}}.
\]
Similarly if $G-a \leq 1$ and $\beta_i$ is the least integer such that $2^{-\beta_i}< G-a - (\sum_{j=1}^{i-1} 2^{-\beta_j})$, then
\[
G\os -k = G - \sum_{i=1}^k \frac{1}{2^{\beta_i}}.
\]
\end{theorem}

\begin{proof}
We proceed by induction on $k$. If $k = 1$, then the result follows by \cref{BalancedNumbersColon1}. Suppose then that $k>1$ and for all $1\leq t< k$, $G:t = G + \sum_{i=1}^t 2^{-\alpha_i}$ and $G:-t = G - \sum_{i=1}^t 2^{-\beta_i}$. Notice that, $\alpha_k$ is the least integer such that $2^{-\alpha_k}< b - G:(k-1) = b - G - \sum_{i=1}^{k-1} 2^{-\alpha_i}$. Thus, \cref{BalancedNumbersColon1} implies 
\begin{center}
\begin{align*}
G\os k 
\cong G\os (k-1)\os 1 
\triangleq \game{G:(k-2)}{b}\os 1\\
= G\os (k-1) + \frac{1}{2^{\alpha_k}}
= G + \sum_{i=1}^k \frac{1}{2^{\alpha_i}}
\end{align*}
\end{center}
as required. Similarly, $\beta_k$ is the least integer such that $2^{-\beta_k}< (G\os -(k-1)) - a = G -a - \sum_{i=1}^{k-1} 2^{-\beta_i} $. Thus, \cref{BalancedNumbersColon1} implies 
\begin{center}
\begin{align*}
G\os -k 
\cong (G\os -(k-1))\os 1 
\triangleq \game{a}{G\os -(k-1)}\os 1\\
= (G\os -(k-1)) - \frac{1}{2^{\beta_k}}
= G - \sum_{i=1}^k \frac{1}{2^{\beta_i}}.
\end{align*}
\end{center}
This concludes the proof.
\end{proof}

\begin{corollary}\label{b-G Special Case}
If $G \cong \game{a}{b}$ such that $a,b,G$ are numbers and $b -G = \frac{1}{2^{p}}$ for some $p$, then
\[
G\os k = G + \frac{1}{2^p} - \frac{1}{2^{p+k}}.
\]
where $k>0$ is an integer.
\end{corollary}

\begin{proof}
Observe that as $b- G = \frac{1}{2^p}$, for all $i\geq 1$, $\alpha_i = p+i$, that is $2^{-(p+i)} < b-G - (\sum_{j=1}^{i-1} 2^{-\alpha_j}) = \frac{1}{2^p} - (\sum_{j=1}^{i-1} 2^{-(n+j)})$. Then by \cref{GameOsInteger}, $G:k = G + \sum_{i=1}^k \frac{1}{2^{p+i}}$. Recalling the well known identity that $\sum_{i=1}^m \frac{1}{2^m} = 1 - 2^m$ notice that,
\begin{center}
\begin{align*}
G:k =  G + \sum_{i=1}^k \frac{1}{2^{p+i}} = G + (\sum_{i=1}^{p+k} \frac{1}{2^{i}}) - (\sum_{j=1}^{p} \frac{1}{2^{j}}) \\
= G + (1-\frac{1}{2^{p+k}}) - (1-\frac{1}{2^p}) = G + \frac{1}{2^p} - \frac{1}{2^{p+k}}
\end{align*}
\end{center}
as required.
\end{proof}

\begin{corollary}\label{G-a Special Case}
If $G \cong \game{a}{b}$ such that $a,b,G$ are numbers and $G-a = \frac{1}{2^{p}}$ for some $p$, then
\[
G\os -k = G - \frac{1}{2^p} + \frac{1}{2^{p+k}}.
\]
where $k>0$ is an integer.
\end{corollary}

\begin{proof}
Proof follows by an almost identical argument to \cref{b-G Special Case}.
\end{proof}

Recall that any non-integer $\frac{a}{2^p}$ in canonical form is identically $\interval{\frac{a}{2^p}}{-\frac{1}{2^p}}$. Thus, the value of any ordinal sum of the form $\frac{a}{2^p}:k$ where $k$ is an integer is given by \cref{b-G Special Case} or \cref{G-a Special Case}. With slightly more work we can also show other identities where $k$ is not an integer.

For example we now show the following identity $\frac{1}{2^p} \os \frac{1}{2^q} = \frac{1}{2^p}+\frac{1}{2^{p+q+1}}$ given in \cite{albert2019lessons} page 240 as a demonstration of \cref{BalancedNumbersColon1} and \cref{G-a Special Case}. Observe that $\frac{1}{2^p} \cong \interval{\frac{1}{2^p}}{-\frac{1}{2^{p}}}$ and $\frac{1}{2^q} \cong 1: -q$. Then applying \cref{BalancedNumbersColon1} and \cref{GameOsInteger},
\begin{center}
\begin{align*}
\frac{1}{2^p} \os \frac{1}{2^q} \cong  \interval{\frac{1}{2^p}}{-\frac{1}{2^{p}}} \os 1 \os -q\\
\triangleq \interval{\frac{1}{2^p}+\frac{1}{2^{p+1}}}{-\frac{1}{2^{p+1}}} \os -q\\
= \frac{1}{2^p}+\frac{1}{2^{p+1}} - \frac{1}{2^{p+1}} + \frac{1}{2^{p+q+1}}\\
= \frac{1}{2^p} + \frac{1}{2^{p+q+1}}.
\end{align*}
\end{center}
It is easy to see that the same argument implies the more general statement that $\frac{a}{2^p} \os \frac{1}{2^q} = \frac{a}{2^p} + \frac{1}{2^{p+q+1}}$ whenever $a$ is odd. In the same vein we can show other identities. For example, we claim that $\frac{a}{2^p} \os (1-\frac{1}{2^q}) = \frac{a}{2^p} + \frac{1}{2^{p+1}} -\frac{1}{2^{p+q+1}}$ where $a$ is odd. Proceeding as before, $\frac{a}{2^p} \cong \interval{\frac{a}{2^p}}{-\frac{1}{2^p}}$ and $1-\frac{1}{2^q} = 1:-1:(q-1)$. Then,
\begin{center}
\begin{align*}
\frac{a}{2^p} \os (1-\frac{1}{2^q}) \cong  \interval{\frac{a}{2^p}}{-\frac{1}{2^p}}\os 1 \os -1\os (q-1)\\
\triangleq \interval{\frac{a}{2^p}+\frac{1}{2^{p+1}} - \frac{1}{2^{p+2}}}{-\frac{1}{2^{p+2}}} \os (q-1)\\
= \frac{a}{2^p}+\frac{1}{2^{p+1}} - \frac{1}{2^{p+2}} + \frac{1}{2^{p+2}} - \frac{1}{2^{p+ 2 + (q-1)}}\\
= \frac{a}{2^p} + \frac{1}{2^{p+1}} - \frac{1}{2^{p+q+1}},
\end{align*}
\end{center}
which is the desired result. With these examples as intuition we now prove a much more general identity that applies to any ordinal sum of numbers where the base is balanced. Note this implies that the following applies whenever the base of the ordinal sum is a non-integer in canonical form.

\begin{theorem}[Balanced Ordinal Sum Theorem]\label{BalancedOSNumber}
If $\interval{n}{-\frac{1}{2^p}}$ is balanced and $n$ is a number, then for all integers $m\geq 0$ and odd integers $0 < a < 2^q$ or integer $a=0$, such that $m+a2^{-q}>0$;
\[
\interval{n}{-\frac{1}{2^p}} \os (m+\frac{a}{2^q}) = n + \frac{1}{2^p} - \frac{1}{2^{p+m}} + \frac{a}{2^{p+m+q+1}}.
\]
\end{theorem}

\begin{proof}
We consider the case $m=0$ and the case $m>0$ distinctly.

\vspace{0.25cm}
\hspace{-0.5cm}
\underline{Case.1:} $m=0$. Hence, if $a = 0$ the result is trivial. Suppose then that $a\neq 0$. We proceed by induction on $q\geq 0$. If $q = 0$, then the result follows directly from \cref{b-G Special Case} or \cref{G-a Special Case}. Suppose then that $q>0$. Then, $\frac{a}{2^q} = \interval{\frac{a}{2^q}}{-\frac{1}{2^q}}$. Thus, the Colon Principle implies that 
\[
\interval{n}{-\frac{1}{2^p}} \os \frac{a}{2^q} = \interval{n}{-\frac{1}{2^p}} \os \interval{\frac{a}{2^q}}{-\frac{1}{2^q}}.
\]
Observe that $\frac{a}{2^q} - \frac{1}{2^q} = \frac{b}{2^{q-1}}$ where $b = \frac{a-1}{2}$ and $\frac{a}{2^q} + \frac{1}{2^q} = \frac{c}{2^{q-1}}$ where $c = \frac{a+1}{2}$. Then by induction,
\[
\interval{n}{-\frac{1}{2^p}} \os \frac{b}{2^{q-1}} = n + \frac{b}{2^{p+q}}
\]
and 
\[
\interval{n}{-\frac{1}{2^p}} \os \frac{c}{2^{q-1}} = n + \frac{c}{2^{p+q}}.
\]
Hence, $\interval{n}{-\frac{1}{2^p}} \os \frac{a}{2^q} \in (n + \frac{b}{2^{p+q}}, n + \frac{c}{2^{p+q}})$. Given our choices of $b$ and $c$ it should be clear that the simplest number of this interval is in fact $n + \frac{a}{2^{p+q+1}}$. Thus, $\interval{n}{-\frac{1}{2^p}} \os \frac{a}{2^q} = n + \frac{a}{2^{p+q+1}}$ as required.

\vspace{0.25cm}
\hspace{-0.5cm}
\underline{Case.2:} $m>0$. Observe that $m+ \frac{a}{2^{q}} = m\os \frac{a}{2^{q}}$. Thus, by the Colon Principle, 
\[
\interval{n}{-\frac{1}{2^p}} \os (m+\frac{a}{2^q}) = \interval{n}{-\frac{1}{2^p}} \os m \os \frac{a}{2^q}.
\]
We consider first, $\interval{n}{-\frac{1}{2^p}} \os m$. By \cref{b-G Special Case}, $\interval{n}{-\frac{1}{2^p}} \os m = n + \frac{1}{2^p} - \frac{1}{2^{p+m}}$. Notice that $\interval{n}{-\frac{1}{2^p}} \os (m-1)$ is the greatest Left option of $\interval{n}{-\frac{1}{2^p}} \os m$ and again by \cref{b-G Special Case}, $\interval{n}{-\frac{1}{2^p}} \os (m-1) = n + \frac{1}{2^p} - \frac{1}{2^{p+m-1}}$. Of course the best Right option of $\interval{n}{-\frac{1}{2^p}} \os m$ is $n + \frac{1}{2^p}$ as Right has gained no new options in the subordinate. Thus, $\interval{n}{-\frac{1}{2^p}} \os m \triangleq \interval{n+\frac{1}{2^p} -\frac{1}{2^{p+m}}}{-\frac{1}{2^{p+m}}}$ which is balanced. Then,
\begin{center}
\begin{align*}
\interval{n}{-\frac{1}{2^p}} \os m \os \frac{a}{2^q} \triangleq \interval{n+\frac{1}{2^p} -\frac{1}{2^{p+m}}}{-\frac{1}{2^{p+m}}} \os \frac{a}{2^q}\\
= n+\frac{1}{2^p} -\frac{1}{2^{p+m}} + \frac{a}{2^{p+m+q+1}}
\end{align*}
\end{center}
by case.1. This concludes the proof.
\end{proof}

Recalling every non-integer $x$ in canonical form is balanced \cref{BalancedOSNumber} implies a formula for the value of an ordinal sum of numbers where the base is in canonical form. This formula must be broken down into several cases, which depend on the signs of the base and subordinate, however, are all easy to verify given what we have already shown. For a list of these formula see Table~3. Note that when reproducing or verifying  Table~3 it is useful to recall that $-(G\os H) = (-G)\os (-H)$ for cases where the base and subordinate have different signs and \cref{GameOsInteger} for the cases where the basis is an integer.

\begin{table}[h!]
\centering
\begin{tabular}{| c | c |}
\hline 
Sum & Value \\ \hline 
$n \os (m + \frac{b}{2^q})$ & $n+m+\frac{b}{2^{q}}$\\ \hline 
$n \os -(m + \frac{b}{2^q})$ & $n - 1 + \frac{1}{2^m} - \frac{b}{2^{m+q+1}}$\\ \hline
$(n + \frac{a}{2^p}) \os (m + \frac{b}{2^q})$ & $n+\frac{a+1}{2^p} - \frac{1}{2^{p+m}} + \frac{b}{2^{p+m+q+1}}$ \\ \hline
$(n + \frac{a}{2^p}) \os -(m + \frac{b}{2^q})$ & $n + \frac{a-1}{2^p} + \frac{1}{2^{p+m}} - \frac{b}{2^{p+m+q+1}}$ \\ \hline
\end{tabular}
\label{tab:CFOrdinalNumbers}
\caption{Ordinal sums of numbers where the base is in canonical form. Here $n,m,a,b \geq 0$ are integers such that $0 < a < 2^p$ and $0 \leq b < 2^q$ are odd.}
\end{table}

\section{Conclusion}

In this paper we have studied ordinal sums of numbers both in general and in a particular well behaved ruleset, which we call \TT{}. This has developed the theory of ordinal sums though presenting known, but not yet published results from \cref{TeqBaseImpOSTeq}, as well as introducing novel pieces of notation such as balls $\interval{m}{\Delta}$ and the notion of a game being balanced. Contributions that are significant given every number is equivalent modulo domination to a ball (\cref{Teq=Number}) and \cref{TeqBaseImpOSTeq} implying that this is sufficient to reduce every ordinal sum of numbers to an ordinal sum of balls.

Along with analysing \TT{} we have also given a formulas for the value of ordinal sums of numbers where the subordinate is an integer, see \cref{GameOsInteger}. We have also improved this for special cases, see \cref{b-G Special Case} and \cref{G-a Special Case}. In particular for numbers that are balanced with radius $\frac{-1}{2^p}$, which includes all canonical forms of non-integers, we have used these results to give a closed formula for an ordinal sum with any subordinate that is a number \cref{BalancedOSNumber}. Thereby resolving the problem of determining the value of an ordinal sum where the basis is in canonical form, originally proposed by Conway in \cite{conway2000numbers} (see pages 192-195), for numbers.

We have also constructed the canonical form of every number as a tower in \TT{} (see \cref{LeaningTower}) where no story (summand) has a different sign then the tower itself. Beyond this we have classified the structure modulo domination of a story in \TT{} through \cref{n-nTeq1-1} and \cref{n-mTeq0+k}. From which the structure of a tower modulo domination can be easily verified in linear time in the number stories by repeatedly applying \cref{GameOsInteger}. Given that it is simple to analyze the form and value of a disjunctive sum of games where the summands are numbers, we have outlined an easy way to determine the winner and value of a \TT{} game (which provides insight into winning strategies).

We conclude the paper with a list of open questions relating to this work;
\begin{enumerate} 
    \item Describe the forms equivalent modulo domination and values of \textsc{Blue-Green-Red} \TT{}, that is, the games of the form $\sum_{j=1}^m \bigodot_{i=1}^n (b_i+g_i+r_i)$ where $b_i$ is a non-negative integer, $r_i$ is a non-positive integer, and $g_i$ is a nimber;
    \item Extend the theory of numbers modulo domination to numbers born on day $\omega$ and later;
    \item Determine if there is an extension of our results to surreal numbers;
    \item Extend \cref{LeaningTower} to numbers born on day $\omega$ and later;
    \item Investigate ordinal sums of numbers of the form $G \cong \game{a}{b}$ where $G-a$ or $b-G$ is strictly greater than $1$;
    \item Using \cref{TeqBaseImpOSTeq} analyze ordinal sums of a class of games that are not numbers or impartial games;
    \item Give criteria for balls $\interval{m}{\Delta}$ to be balanced when $m,\Delta$ and/or $\interval{m}{\Delta}$ are not numbers;
    \item Letting $\interval{m}{\Delta_1, \dots, \Delta_k} \cong \game{m+\{\Delta_i\}}{m-\{\Delta_i\}}$ be a \emph{generalized ball}, prove or disprove that for all normal play games $G$, there exists a $\interval{m}{\Delta_1, \dots, \Delta_k}$ such that $G = \interval{m}{\Delta_1, \dots, \Delta_k}$.
\end{enumerate}

\section*{Acknowledgements}
We would also like to acknowledge the support of the Natural Sciences and Engineering Research Council of Canada (NSERC) for support through the Canadian Graduate Scholarship -- Master's program.

\bibliographystyle{plain}
\bibliography{towers}

\end{document}